\documentclass[11pt]{article}
\usepackage{graphicx}
\usepackage{amsmath}
\usepackage{amssymb}
\usepackage{mathptmx}
\usepackage{graphicx}
\usepackage{esint}
\usepackage{xcolor}
\usepackage{algpseudocode}
\usepackage{url}
\usepackage{comment}
\usepackage[active]{srcltx}

\usepackage{tikz}
\usepackage{pgfplots}
\pgfplotsset{compat=1.15}
\usepackage{mathrsfs}
\usetikzlibrary{arrows}
\usepackage{circuitikz}
\usepackage[normalem]{ulem}
\usetikzlibrary{tikzmark,decorations.pathmorphing,arrows,patterns}

\usepackage{latexsym, amsmath, amsfonts, amscd, amssymb, verbatim}
\usepackage{latexsym, amsmath, amsfonts, amscd, amssymb, verbatim}

\usepackage[margin=3cm]{geometry} 

\usepackage{color}
\usepackage{graphicx}

\usepackage{amsthm}

\theoremstyle{plain}
\newtheorem{Theorem}{Theorem}[section]
\newtheorem{Proposition}[Theorem]{Proposition}
\newtheorem{Lemma}[Theorem]{Lemma}

\theoremstyle{definition}

\newtheorem{Remark}[Theorem]{Remark}
\normalsize

\newcommand{\HH}{\mathcal H}
\newcommand{\weakto}{\rightharpoonup}
\newcommand{\norm}[1]{\lVert #1\rVert}
\newcommand{\ip}[2]{\langle #1,#2\rangle}
\begin{document}

\title{
The Popov's  Algorithm with
Optimal Bounded Stepsize for 
%Generalized Monotone Equation %
Variational Inequalities
%\footnote{Dedicate to Professor Simeon Reich  on the occasion of his 80th birthday.}
}

	\author{Nhung Hong  Nguyen\thanks{Faculty of Applied Sciences, HCMC University of Technology and Engineering, Ho Chi Minh City, Vietnam. Email:nhungnh@hcmute.edu.vn}, 
    \quad  Thanh Quoc Trinh\thanks{Faculty of Basic Sciences, Van Lang University, Ho Chi Minh City, Vietnam. Email: thanh.tq@vlu.edu.vn},  
    \quad
    Phan Tu Vuong\thanks{School of Mathematical Sciences,
		University of Southampton, SO17 1BJ, Southampton, UK. Email: t.v.phan@soton.ac.uk    } }
\maketitle

\medskip
\begin{quote}
\noindent {\bf Abstract.} 
{\color{black}
For solving unconstrained generalized monotone variational inequality, we prove that the iterations generated by Popov algorithm converge weakly to a solution when stepsize is smaller than $\frac{1}{\sqrt{3}L}$, where $L$ is the Lipschitz constant of the variational inequality operator. The convergence analysis is carried out by using a new Lyapunov-type function. In addition, we prove that the upper bound $\frac{1}{\sqrt{3}L}$ is tight. Finally, we extend the results to variational inequality with linear subspace constraints.
%We provide an example showing that if the weak continuity assumption is violated, the Popov algorithm does not converges with stepsize $\frac{1}{2L}$ .
}

\medskip
\noindent {\bf Mathematics Subject Classification (2010).}\ 47J20,
49J40, 49M30.

\medskip
\noindent {\bf Keywords.} Popov's algorithm, Optimal stepsize, Generalized monotonicity, Variational inequality. 
\end{quote}

\section{Introduction}
Let $\HH$ be a real Hilbert space with the inner product $\left< \cdot,\cdot\right>$ and a generated norm $\|\cdot\|$. 
Let $K \subset \HH$ be a nonempty, closed and convex set, let $F$ be a continuous operator from $\HH$ to $\HH$. 
In this paper, we revisit some projection algorithms for solving the variational inequality: Find $u^* \in K$ such that
\begin{equation}  \label{vi}
     \langle F(u^*), u-u^* \rangle \ge 0 \quad \forall u \in K. 
\end{equation}  
  When $K=\HH$, the variational inequality $VI(F,K)$ reduces to the nonlinear
equation: Find $u^*\in \HH$ such that
\begin{equation}\label{ME}
	F(u^*)=0.
\end{equation} 
For a comprehensive study of variational inequality and its applications, we refer the readers to excellent books \cite{FacchineiPang03,KinderlehrerStampacchia80}.

One of the oldest but popular methods for solving variational inequality \eqref{vi} is the extragradient algorithm proposed by Korpelevich \cite{Korpelevich1976}, which generates the iterative sequences by 
\begin{equation}\label{Ko}
    \begin{cases}
  v^{k} &= P_K (u^k -\lambda F(u^k)),\\
  u^{k+1} &= P_K (u^{k}-\lambda F(v^k)),  
\end{cases}
\end{equation}
with $u^0 \in \HH$ and stepsize $\lambda \in \left( 0, \frac{1}{L} \right)$, where $L$ is the modulus of Lipschitz continuity of the {\it monotone} operator $F$ and $P_K$ is the projection onto the closed convex set $K$. It is well understood that the range of stepsize significantly affects the convergence speed. Typically, the larger stepsize provides faster convergence. 
However, if the stepsize is too large then the algorithm may diverge. 
Hence, one of the interesting research questions is to find the maximal upper bound for the stepsize. It was proved in \cite{Khanh16,Xu} that the upper bound $\frac{1}{L}$ of the extragradient algorithm \eqref{Ko} is tight, i.e., cannot be enlarged without losing the convergence.

One notable variation of the extragradient algorithm for solving monotone inequality is the Tseng's forward-backward-forward algorithm \cite{Tseng}, which requires only one projection
\begin{equation}\label{Ts}
    \begin{cases}
  v^{k} &= P_K (u^k -\lambda F(u^k)),\\
  u^{k+1} &= v^{k}-\lambda F(v^k)+ \lambda F(u^k).  
\end{cases}
\end{equation}
Another variation of the extragradient algorithm is the extragradient subgradient algorithm proposed in \cite{Gibali}, where the authors replaced the second projection in the extragradient algorithm by a suitable projection onto a half space. 
The convergence analysis of the Tseng's algorithm \cite{Tseng} and the extragradient subgradient algorithm  \cite{Gibali} also requires the stepsize $\lambda \in \left( 0, \frac{1}{L} \right)$. 
The convergence of these extragradient-type algorithms in Hilbert spaces with {\it pseudo-monotone} operator was obtained in \cite{BCV,PTVuong2018}.
Since the extragradient algorithm \eqref{Ko} coincides with the Tseng's algorithm and the extragradient subgradient algorithm in the unconstrained case, examples given in \cite{Khanh16,Xu} also showed that the upper bound $\frac{1}{L}$ is tight for these algorithms. 

Popov's algorithm proposed in \cite{Popov1980} is an improved version of extragradient algorithm  with an advantage that  each iteration requires only one operator evaluation instead of two as in the extragradient-type algorithms. The Popov's algorithm computes the iterative sequences as   
\begin{equation} \label{PP}
    \begin{cases}
  u^{k+1} = P_K (u^k -\lambda F(v^k)),\\
  v^{k+1} = P_K (u^{k+1}-\lambda F(v^k)).  
\end{cases}
\end{equation}
with $u^0, v^0 \in \HH$ and $\lambda >0$. 
To guarantee convergence of the Popov  algorithm \eqref{PP}, the bound $\lambda <\frac{1}{3L}$ was required in the original analysis \cite{Popov1980}. This upper bound has been enlarged to $\frac{\sqrt{2}-1}{L}$ in \cite{Hieu,Xu} and then $\frac{1}{2L}$ (see \cite{Vuong}). 
These upper bounds are considerablely smaller than the optimal upper bound $\frac{1}{L}$ of the extragradient-type algorithms. It is still open if the upper bound $\frac{1}{2L}$ of the Popov's algorithm is tight or not. 

One can also combine the two projections in the Popov's algorithm to obtain a single projection single operator evaluation algorithm

\begin{equation}\label{FrB}
  u^{k+1} = P_K (u^k -\lambda (2F(u^k)-F(u^{k-1}))),  
\end{equation}
which is called forward-reflected-backward algorithm \cite{Malitsky2}.
%or the past extragradient algorithm \cite{Axel}. 
Again, the convergence analysis requires stepsize $\lambda \in (0,\frac{1}{2L})$ and it is still  unclear whether the upper bound $\frac{1}{2L}$ is tight or not.

{\color{black}

In the unconstrained case, i.e. $K=\HH$, the Popov algorithm \eqref{PP} coincides with the projected reflected gradient method \cite{Malitsky} and the optimistic-gradient descent-ascent (OGDA) (see \cite{Daskalakis}), which reads as
$$
u^{k+1} = u^k -\lambda F(v^k), \quad  v^{k+1} = 2u^{k+1}-u^k,
$$
or equivalently 
\begin{equation}\label{PPu}
  u^{k+1} = u^k -\lambda F(2u^k-u^{k-1}). 
\end{equation}
If, in addition, $F$ is a linear operator, then the Popov algorithm also coincides with the forward-reflected-backward algorithm \eqref{FrB}. \\

%In the first contribution, 

In this technical note,
we prove that the Popov's algorithm still converges with  stepsize $\lambda \in (0,\frac{1}{\sqrt{3}L})$ for unconstrained case. More importantly, we show that the upper bound $\frac{1}{\sqrt{3}L}$ is tight by constructing an example confirming that the Popov's algorithm \eqref{PPu} does not converge with stepsize $\lambda = \frac{1}{\sqrt{3}L}$. We also  extend the obtained results to variational inequality with linear subspace constraint.\\

%The second contributation is to prove that the classical upper bound $\frac{1}{2L}$ of the stepsize in Popov algorithm \eqref{PP} is indeed tight. Specifically,  we will construct an example of operator $F$ in Hinbert spaces such that the algorithm does not converge with stepsize $\lambda = \frac{1}{2L}$ when the weak continuity Assumption is violated. 
%Our example is applied only for the constrained case, i.e. $K$ is a real subset of $\HH$. 
}
\section{Main results}
We make the following common assumptions on problem \eqref{vi}:
\begin{itemize}
    \item[(A1)] The solution set of problem \eqref{ME} denoted by $S^*$ is nonempty;
    \item[(A2)] $F\colon \HH \to \HH$ is $L$--Lipschitz-continuous for some $L>0$; 
    \item[(A3)] The operator $F$ is quasar-monotone, i.e. 
    $$
    \langle F(u),u-u^*\rangle \ge 0 \quad \forall u^* \in S^*, \ \forall u\in \HH.
    $$   
    \item[(A4)] $F$ is sequentially weakly continuous, i.e.
  \[
    u_n\weakto u\quad\Longrightarrow\quad F(u_n)\weakto F(u).
  \]
  Here $u_n\weakto u$ means that
  $\langle u_n,h \rangle \to \langle u,h \rangle$ for every $h\in \HH$.  
\end{itemize}
Assumption (A3) is strictly weaker than
(pseudo)-monotonicity \cite{NhungThanhVuong}. Recall that $F$ is
pseudo-monotone if
\begin{equation*}
  \langle F(v), u-v \rangle \geq 0 \quad \Longrightarrow \quad \langle F(u), u-v \rangle \ge 0\quad \forall u,v \in \HH,
\end{equation*}
and $F$ is monotone if $\langle F(u)-F(v), u-v\rangle \ge 0 \quad \forall u, v \in \HH$.
{\color{black} When $F$ is monotone, (A4) is not required in the convergence analysis.}
\\
We use the following Hilbert space form of Opial's lemma. % \cite{}.

\begin{Lemma}[Opial]
\label{lem:Opial}
Let $S$ be a nonempty subset of a real Hilbert space $\HH$, and let
$\{u^k\}\subset \HH$. Assume that
\begin{enumerate}
 \item for every $u^*\in S$, the limit $\lim_{k\to\infty}\norm{u^k-u^*}$
 exists and is finite;
 \item every weak sequential cluster point of $\{u^k\}$ belongs to $S$.
\end{enumerate}
Then $u^k\weakto\bar u$ for some $\bar u\in S$.
\end{Lemma}
\subsection{Larger stepsize and tight upper bound}
We establish some properties of the iterative sequences $\{u^k\}$ and $\{v^k\}$ generated by unconstrained Popov algorithm \eqref{PPu}. Firstly, we rewrite the iteration in \eqref{PPu} 
\begin{equation} \label{eq:Popov}
    \begin{cases}
  u^{k+1} = u^k -\lambda F(v^k),\\
  v^{k+1} = u^{k+1}-\lambda F(v^k).  
\end{cases}
\end{equation}
This gives, for every $k\ge1$,
\begin{align}
 v^k&=u^k-\lambda F(v^{k-1})=2u^k-u^{k-1},
 \label{eq:reflection}\\
 v^k-u^k&=-\lambda F(v^{k-1})=u^k-u^{k-1},
 \label{eq:e-identity}\\
 v^k-v^{k-1}&=(u^k-v^{k-1})+(v^k-u^k),
 \label{eq:delta-identity}\\
 u^{k+1}-v^k&=-\lambda\bigl(F(v^k)-F(v^{k-1})\bigr).
 \label{eq:r-identity}
\end{align}
Let
$\alpha = \lambda^2 L^2$, the Lipschitz continuity of $F$ yields
\begin{equation}
 \norm{u^{k+1}-v^k}^2
 \le\alpha\norm{v^k-v^{k-1}}^2\qquad(k\ge1).
 \label{eq:r-bound1}
\end{equation}
Besides, for $u^*\in S^*$, then $F(u^*)=0$. Hence we will also repeatedly use the following inequality
\begin{equation} \label{eq:root-Lipschitz1}
 \norm{F(v)}\le L\norm{v-u^*}
 \qquad (v\in \HH,\ u^*\in S^*).
\end{equation}

We start with a result on a new Lyapunov energy function. 
\begin{Proposition}\label{SequenceProperty} 
Suppose that the Assumptions (A1-A3) hold.
For any solution $u^* \in S^*$ and  for every  $k \ge 2 $, 
\begin{equation} \label{eq:A-def}
    A_k (u^*) := \|u^k-u^*\|^2-\|u^{k-1} - v^{k-1}\|^2+2\alpha \|v^{k-1} - v^{k-2}\|^2
\end{equation}
and 
\begin{equation} \label{eq:B-def}
    B_k:= (1-3\alpha)\|u^{k+1}-v^k\|^2
+\frac{2}{3}(1-3\alpha)\|v^{k}-v^{k-1}\|^2,
\end{equation}
then 
\begin{equation}\label{Main_Inequality}
A_{k+1} (u^*)\le A_k (u^*)- B_k.
\end{equation}
\end{Proposition}

\begin{proof}
Let $u^* \in S^*$,   
since $u^{k+1}=u^k-{\lambda}F(v^k)$ we can write
\begin{align}\label{mainEst}
\begin{split}
\|u^{k+1}-u^*\|^2
&=\|u^k-\lambda{F(v^k)}-u^*\|^2-\|u^k-\lambda F(v^k)-u^{k+1}\|^2\\
&=\|u^k-u^*\|^2-\|u^k-u^{k+1}\|^2-2\lambda\langle F(v^k),u^{k+1}-u^*\rangle\\
&=\|u^k-u^*\|^2-\|u^k-v^k\|^2-\|v^k-u^{k+1}\|^2-2\langle u^k-v^k,v^k-u^{k+1}\rangle\\
&\quad-2\lambda\langle F(v^k),u^{k+1}-u^*\rangle.
\end{split}
\end{align}
By assumption (A3)
$$
\langle F(v^k),v^k-u^*\rangle\geq 0, 
$$
which implies
\begin{equation}\label{pseudo}
\langle F(v^k),v^k-u^{k+1}\rangle\geq -\langle F(v^k),u^{k+1}-u^*\rangle.
\end{equation}
Combining \eqref{pseudo} with \eqref{mainEst} we get
\begin{align}\label{mainv}
	\|u^{k+1}-u^*\|^2&\leq
	\|u^k-u^*\|^2-\|u^k-v^{k}\|^2-\|v^k-u^{k+1}\|^2-2\langle u^k-\lambda F(v^k)-v^k,v^k-u^{k+1}\rangle\notag\\
	&=\|u^k-u^*\|^2-\|u^k-v^k\|^2-\|v^k-u^{k+1}\|^2+2\lambda \langle F(v^k)-F(v^{k-1}),v^k-u^{k+1}\rangle.
\end{align}
We estimate the last term 
in the above inequality as follows
\begin{align}\label{mainv2}
	2\lambda \langle F(v^k)-F(v^{k-1}),v^k-u^{k+1}\rangle 
	&\leq 2\lambda\|F(v^k)-F(v^{k-1})\|\|v^k-u^{k+1}\|\notag\\
	&\leq 2\lambda L\|v^k-v^{k-1}\|\|v^k-u^{k+1}\|\notag\\
	&\leq \frac{1}{3} \|v^k-v^{k-1}\|^2 
	+ 3 \lambda^2 L^2\|v^k-u^{k+1}\|^2\notag\\
 &= \frac{1}{3} \|v^k-v^{k-1}\|^2 
	+ 3 \alpha\|v^k-u^{k+1}\|^2.
\end{align}
%{\color{red} Up to here, it holds for constrained VIP as well.\\}
On the other hand we have
\begin{eqnarray} \label{break}
 \|v^k-v^{k-1}\|^2 
 &= & 2 \|u^k-v^{k-1}\|^2 -  \|u^{k-1}-v^{k-1}\|^2 +2 \|u^k-u^{k-1}\|^2  \\\nonumber
  &= & 2 \|(u^{k-1}-\lambda F(v^{k-1}))-(u^{k-1}-\lambda F(v^{k-2}))\|^2 -  \|u^{k-1}-v^{k-1}\|^2 +2 \|u^k-v^{k}\|^2  \\\nonumber
  &\le & 2 \lambda^2\| F(v^{k-1})- F(v^{k-2})\|^2 -  \|u^{k-1}-v^{k-1}\|^2 +2 \|u^k-v^{k}\|^2  \\\nonumber
 &\le & 2 \alpha \|v^{k-1}-v^{k-2}\|^2 -  \|u^{k-1}-v^{k-1}\|^2 +2 \|u^k-v^{k}\|^2.\nonumber
\end{eqnarray}
%{\color{red} What we do not have in the constrained case is the first equality above!\\}
Hence 
\begin{eqnarray} \label{mainv4}
 \frac{1}{3} \|v^k-v^{k-1}\|^2 
 &=& \|v^k-v^{k-1}\|^2 - \frac{2}{3} \|v^k-v^{k-1}\|^2\\
 &\le & 2 \alpha \|v^{k-1}-v^{k-2}\|^2 -  \|u^{k-1}-v^{k-1}\|^2 +2 \|u^k-v^{k}\|^2 - \frac{2}{3} \|v^k-v^{k-1}\|^2. \nonumber
\end{eqnarray}
Combining \eqref{mainv}, \eqref{mainv2} and \eqref{mainv4} we obtain
\begin{equation*}\label{mainv3}
	\begin{split}
		\|u^{k+1}-u^*\|^2
  &\leq \|u^k-u^*\|^2
		-(1-3\alpha)\|v^k - u^{k+1}\|^2 -\|u^{k-1}-v^{k-1}\|^2\\
    &+\|u^k-v^k\|^2+2\alpha \|v^{k-1}-v^{k-2}\|^2
    - \frac{2}{3} \|v^k-v^{k-1}\|^2,
	\end{split}
\end{equation*}
which implies 
\eqref{Main_Inequality}.
\qed
\end{proof}
\begin{Remark}
    For constrained case, the estimations in the proof of Proposition \ref{SequenceProperty} still hold up to inequality \eqref{mainv2}. The analysis is broken down at equality \eqref{break}. 
\end{Remark}

\begin{Remark}
  By a similar argument, we can prove that inequality \eqref{Main_Inequality} also holds with
$$
A_k (u^*) := \|u^k-u^*\|^2+ 3\alpha \|u^{k} - v^{k-1}\|^2+3\alpha^2 \|v^{k-1} - v^{k-2}\|^2 - 3\alpha \|u^{k-1} - v^{k-1}\|^2
$$
and 
$$
B_k:= (1-3\alpha)\|u^{k+1}-v^k\|^2
+(1-3\alpha)\|u^{k}-v^{k}\|^2 +\alpha (1-3\alpha)\|v^{k}-v^{k-1}\|^2.
$$
\end{Remark}
We are now in a position to present the first main result of this section. 
From \eqref{Main_Inequality}, to guarantee that the sequence $\{A_k (u^*)\}$ is non-increasing, we need to require that $B_k$ is positive, or equivalently $\lambda < \frac{1}{\sqrt{3}L}$. 
The sequence $\{A_k (u^*)\}$ will play an important role  as a Lyapunov-type function. 
{\color{black}
We start with a key lower bound for $\{A_k(u^*)\}$.
\begin{Proposition}
   Suppose that Assumptions (A1)-(A3) hold.  If $\lambda < \frac{1}{\sqrt{3}L}$, then the sequence $\{A_k (u^*)\}$ is non-negative for any $u^* \in S^*$.
\end{Proposition}

\begin{proof}

Fix $u^* \in S^*$ and $k\ge2$. By \eqref{eq:r-bound1} at index $k-1$,
\[
 \norm{u^k-v^{k-1}}^2
 \le\alpha\norm{v^{k-1}-v^{k-2}}^2.
\]
The definition of $A_k(u^*)$ therefore gives
\begin{equation}
 \begin{aligned}
 A_k(u^*)\ge{}&\norm{u^k-u^*}^2-\norm{u^{k-1}-v^{k-1}}^2\\
            &+2\norm{u^k-v^{k-1}}^2.
 \end{aligned}
 \label{eq:A-first-bound}
\end{equation}
Replacing $u^{k-1}=u^k+\lambda F(v^{k-1})$ and expansion of the
squared norms yield
\begin{equation}
 \begin{aligned}
 &-\norm{u^{k-1}-v^{k-1}}^2+2\norm{u^k-v^{k-1}}^2\\
 &\qquad=-\norm{u^k-v^{k-1}+\lambda F(v^{k-1})}^2
          +2\norm{u^k-v^{k-1}}^2\\
 &\qquad=\norm{u^k-v^{k-1}-\lambda F(v^{k-1})}^2
          -2\lambda^2\norm{F(v^{k-1})}^2.
 \end{aligned}
 \label{eq:square-identity}
\end{equation}
Consequently,
\begin{equation}
 \begin{aligned}
 A_k(u^*)\ge{}&\norm{u^k-u^*}^2
       +\norm{u^k-v^{k-1}-\lambda F(v^{k-1})}^2\\
       &-2\lambda^2\norm{F(v^{k-1})}^2.
 \end{aligned}
 \label{eq:A-quadratic-bound}
\end{equation}
By \textbf{(A3)} and \eqref{eq:root-Lipschitz1},
\[
 \ip{v^{k-1}-u^*}{\lambda F(v^{k-1})}\ge0,
 \qquad
 \lambda^2\norm{F(v^{k-1})}^2
       \le\alpha\norm{v^{k-1}-u^*}^2.
\]
It follows that
\begin{align*}
 &\alpha\norm{v^{k-1}-u^*+\lambda F(v^{k-1})}^2\\
 &\qquad\ge\alpha\norm{v^{k-1}-u^*}^2
           +\lambda^2\alpha\norm{F(v^{k-1})}^2\\
 &\qquad\ge(1+\alpha)\lambda^2\norm{F(v^{k-1})}^2.
\end{align*}
Thus
\begin{equation}
 \lambda^2\norm{F(v^{k-1})}^2
 \le\frac{\alpha}{1+\alpha}
       \norm{v^{k-1}-u^*+\lambda F(v^{k-1})}^2.
 \label{eq:F-bound}
\end{equation}
Substituting \eqref{eq:F-bound} into \eqref{eq:A-quadratic-bound}
and expanding gives
\begin{equation}
 \begin{aligned}
 A_k(u^*)
 &\ge\norm{u^k-u^*}^2
      +\norm{u^k-v^{k-1}-\lambda F(v^{k-1})}^2\\
 &\quad-\frac{2\alpha}{1+\alpha}
       \norm{v^{k-1}-u^*+\lambda F(v^{k-1})}^2\\
 &=2\norm{u^k-u^*}^2
      -2\ip{u^k-u^*}{v^{k-1}-u^*+\lambda F(v^{k-1})}\\
 &\quad+\frac{1-\alpha}{1+\alpha}
       \norm{v^{k-1}-u^*+\lambda F(v^{k-1})}^2.
 \end{aligned}
 \label{eq:quadratic-expansion}
\end{equation}
Completing the square in the last expression yields
 \begin{align} \label{eq:A-lower-bound}
 A_k(u^*)
 &\ge\frac{1-\alpha}{1+\alpha}
       \norm{v^{k-1}-u^*+\lambda F(v^{k-1})
         -\frac{1+\alpha}{1-\alpha}(u^k-u^*)}^2\notag\\
 &\quad+\frac{1-3\alpha}{1-\alpha}\norm{u^k-u^*}^2 \notag\\
 &\ge\frac{1-3\alpha}{1-\alpha}\norm{u^k-u^*}^2.
 \end{align}

Since $0<\alpha<1/3$, both coefficients on the right-hand side of the first inequality are positive. This completes the proof.
%\eqref{eq:A-lower-bound}.
\end{proof}

\begin{Proposition}\label{prop:boundedness}
    Suppose that the Assumptions (A1)-(A3) hold. If $\lambda < \frac{1}{\sqrt{3}L}$, then the sequences $\{u^k\}$ and $\{v^k\}$ generated by Popov's Algorithm \eqref{PPu}  are bounded. More precisely, for every $u^* \in S^*$,
\begin{equation}
\norm{u^k-u^*}^2
       \le\frac{1-\alpha}{1-3\alpha}A_2(u^*)
       \qquad(k\ge2).
 \label{eq:uniform-bound}
\end{equation}
\end{Proposition}

\begin{proof}
    Fix $u^* \in S^*$. Combining 
\eqref{Main_Inequality} and~\eqref{eq:A-lower-bound} gives
\[
 0\le\frac{1-3\alpha}{1-\alpha}\norm{u^k-u^*}^2\le A_k(u^*)\le A_2(u^*)
 \qquad(k\ge2).
\]
Since $0<\alpha<1/3$, this yields \eqref{eq:uniform-bound}. The number
\[
 A_2(u^*)=\norm{u^2-u^*}^2-\norm{u^1-v^1}^2
                            +2\alpha\norm{v^1-v^0}^2
\]
is finite and nonnegative.
For an explicit bound including the initial terms, set
\[
M_{u^*}:=\max\left\{\norm{u^0-u^*},\ \norm{u^1-u^*},\
                   \sqrt{\frac{1-\alpha}{1-3\alpha}A_2(u^*)}\right\}.
\]
Then $\norm{u^k-u^*}\le M_{u^*}$ for all $k\ge0$. By
\eqref{eq:reflection},
\[
 \norm{v^k-u^*}
 =\norm{2(u^k-u^*)-(u^{k-1}-u^*)}\le3 M_{u^*}\qquad(k\ge1).
\]
Together with the fixed initial point $v^0$, this proves boundedness
of both full sequences.
\end{proof}

\begin{Theorem}\label{Mainweakcon}
Suppose that the Assumptions (A1)-(A4) hold.
Let $\lambda < \frac{1}{\sqrt{3}L}$ and $\{u^k\}, \{v^k\} $ are generated by the Popov's algorithm \eqref{PPu}. Then $\{u^k\}$, $\{v^k\}$ converge weakly to a solution of \eqref{ME}.
Furthermore,
\begin{equation} \label{eq:residual-convergence}
 u^{k+1}-u^k\to0,\qquad v^k-u^k\to0,\qquad
 F(u^k) \to 0,\qquad F(v^k)\to0
\end{equation}
in norm. 
%Boundedness of either sequence is not an additional assumption.
\end{Theorem}
    }

%Under the assumptions made and Proposition \ref{SequenceProperty}, the sequence $\{A_k\}$ is decreasing and bounded, and hence is convergent. As a consequence, we can deduce $\lim_{k \to \infty} B_k = 0$. The reminder of the proof is similar to that of \cite[Theorem 3.1]{Vuong} and we skip it for brevity.

%{\color{red} Maybe need to recall the detailed proof.}

{\color{black}
\begin{proof}
Boundedness of both sequences is already supplied by
Proposition~\ref{prop:boundedness}.

\medskip\noindent
Fix $u^* \in S^*$. Summing \eqref{Main_Inequality} from $k=2$ to $N$ and
using $A_{N+1}(u^*)\ge0$, we obtain
\begin{equation}
 \sum_{k=2}^{N}B_k\le A_2(u^*)-A_{N+1}(u^*)\le A_2(u^*).
 \label{eq:telescoping}
\end{equation}
Letting $N\to\infty$ and using \eqref{eq:B-def} gives
\begin{equation}
 \sum_{k=2}^{\infty}\norm{u^{k+1}-v^k}^2
 +\frac23\sum_{k=2}^{\infty}\norm{v^k-v^{k-1}}^2
 \le\frac{A_2(u^*)}{1-3\alpha}<\infty.
 \label{eq:summability}
\end{equation}
Thus $u^{k+1}-v^k\to0$ and $v^k-v^{k-1}\to0$ in norm.
Furthermore
\[
 v^k-u^k=(v^k-v^{k-1})-(u^k-v^{k-1})\to0,
\]
and
\[
 u^{k+1}-u^k=v^{k+1}-u^{k+1}\to0.
\]
Consequently, the iteration and Lipschitz continuity yield
\[
 F(v^k)=\frac{u^k-u^{k+1}}{\lambda}\to0,
 \qquad
 \norm{F(u^k)}\le\norm{F(v^k)}+L\norm{u^k-v^k}\to0.
\]
This proves \eqref{eq:residual-convergence}.

\medskip\noindent
For every $u^* \in S^*$, \eqref{Main_Inequality}
and \eqref{eq:A-lower-bound} show that $A_k(u^*)$ has a finite
nonnegative limit. Rearranging \eqref{eq:A-def} gives
\begin{equation} \label{eq:distance-from-A}
 \norm{u^k-u^*}^2={}A_k(u^*)+\norm{u^{k-1}-v^{k-1}}^2 
                 -2\alpha\norm{v^{k-1}-v^{k-2}}^2.
\end{equation}
The last two terms tend to zero, hence
\[
 \lim_{k\to\infty}\norm{u^k-u^*}^2
 =\lim_{k\to\infty}A_k(u^*)
 \quad\text{exists and is finite for every }u^* \in S^*.
\]

\medskip\noindent
The bounded sequence $\{u^k\}$ has a weakly convergent subsequence in $\HH$. Consider any subsequence satisfying
$u^{k_j}\weakto p$. The weak continuity of $F$ implies
\[
 F(u^{k_j})\weakto F(p).
\]
On the other hand, \eqref{eq:residual-convergence} gives $F(u^{k_j})\to0$. Uniqueness of the weak limit therefore yields $F(p)=0$,
so $p\in S^*$. Thus every weak sequential cluster point lies in $S^*$.

\medskip\noindent
Since $S^*\ne\emptyset$, the Opial's lemma implies $u^k\weakto\bar u$ for some $\bar u\in S^*$.
Finally, $v^k-u^k\to0$ in norm, so for every $h\in \HH$,
\[
 \ip{v^k-\bar u}{h}
 =\ip{u^k-\bar u}{h}+\ip{v^k-u^k}{h}\to0.
\]
Therefore, $v^k\weakto\bar u$ as well.
\qed
\end{proof}
}
\begin{Remark}
    As discussed in \cite[Remark 3.4]{Vuong}, the weak sequential continuity of $F$ can be relaxed. Moreover, it can be removed if $F$ is monotone. In addition, when $F$ is strongly pseudo-monotone, one can deduce the linear convergence of $\{u^k\}$ to the unique solution $u^*$ thanks to an error bound ( see \cite[Proposition 4.1 ]{Vuong}).
\end{Remark}

%\subsection{The tightness of the upper bound $\frac{1}{\sqrt{3}L}$}

\newpage
In the next result, we will provide an example to show that the upper bound  $\frac{1}{\sqrt{3}L}$ of the stepsize $\lambda$ cannot be enlarged.
\begin{Theorem}
The upper bound 
$\frac{1}{\sqrt{3}L}$ in the Popov's Algorithm is tight.
\end{Theorem}
\begin{proof}
    Let $K=\HH=\mathbb{R}^2$ and $F(x_1, x_2) = (-x_2, x_1)$ which is monotone and Lipschitz continuous with $L=1$. The problem \eqref{ME} has unique solution $u^* = (0,0)$. 
The Popov iterations can be written as 
\[
u^{k+1}=
u^k-\lambda\,M\bigl(2 u^k - u^{k-1}\bigr),
\qquad
M=\begin{pmatrix}0&-1\\1&0\end{pmatrix},
\]
with initial data \(u^0, u^1
\not= 0\).  We will show that when 
\(\lambda = 1/\sqrt3\),
$\|u^k\| \not \to 0$.
Identify \(\mathbb{R}^2\) with \(\mathbb{C}\) via
\[
(x_1,x_2)\;\longleftrightarrow\;z = x_1 + i x_2.
\]
Under this identification the operator \(M\) acts by multiplication by \(i\).  
%Writing
%\[z_n \;=\; x_n^{(1)} \;+\; i\,x_n^{(2)},\]
The Popov iteration becomes the scalar complex recursion
\[
z^{k+1}
\;=\;
(1 - 2\,i\lambda)\,z^k
\;+\;
i\lambda\,z^{k-1},
\]
which has a characteristic polynomial 
\[
\mu^2 \;-\;(1-2\,i\lambda)\,\mu \;-\;i\lambda
\;=\;0.
\]
Its 
%discriminant is
%\[
%\Delta
%=\bigl(1-2\,i\lambda\bigr)^2 + 4\,i\lambda
%=1 - 4\lambda^2 \in\mathbb{R}.
%\]
%Hence the two 
roots are
\[
\mu_{1,2}
=\frac{\,1 - 2\,i\lambda \;\pm\;\sqrt{\,1 - 4\lambda^2\,}\,}{2}.
\] 
It is a standard fact 
%(see, e.g., Elaydi, \emph{An Introduction to Difference Equations}, Theorem 2.35, 2.36)
that \(\|u^k\|\to0\) for every choice of initial data if and only if
\[
\max\!\bigl\{|\mu_1|,\;|\mu_2|\bigr\}\;<\;1.
\]
{\color{black}
$$
\mu_1 = \frac{\,1 - i\sqrt3}{2}, \quad \mu_2 = \frac{\,1 - i/\sqrt3}{2},
$$
}
Hence \(\max\{|\mu_1|,|\mu_2|\}=|\mu_1|=1\).  
Therefore, the general solution
\[
z^k = A\,\mu_1^k + B\,\mu_2^k
\]
{\color{black}
never converges to zero if $A\not= 0$ and $B=0$. Equivalently, for \(\lambda=1/\sqrt3\), there are nonzero initial data \(u^0,u^1\) such that the norms \(\|u^k\|\) do not converge to $0$.}

In the figure below, we plot the iterations generated by the Popov algorithm with $\lambda = \frac{1}{\sqrt{3}L}$ and $u^0 = v^0 = (5,5)$ for $1000$ iterations.
It is clear that the Popov's iterations do not converge to the solution $u^*=0$.

\begin{figure}[ht]
\centering
\includegraphics {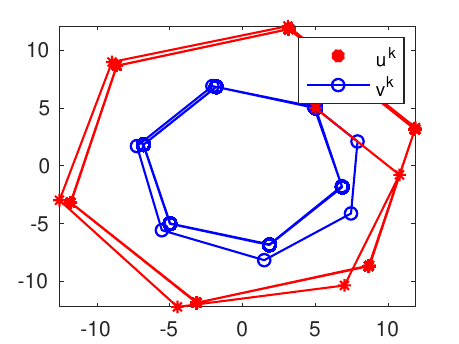}\\ 
\label{fig1}
\end{figure}

\qed
\end{proof}

{\color{black}
\subsection{Linear subspace constrained VI}
In this sub-section, we discuss the extension to constrained case where $K$ is a linear subspace.
 Let $K\subseteq \HH$ be a closed linear subspace.
 %, not necessarily equal to $\HH$. 
 Consider the variational inequality problem VI($F,K$) formed in \eqref{vi}. We show that it can be reduced to an unconstrained problem. In this subsection we use the notation
     \[
     S^*_K:=\{u^*\in K:\ip{F(u^*)}{u-u^*}\ge0\ \text{for every }u\in K\}
          =\{u^*\in K:P_KF(u^*)=0\}.
     \]
We make the analogue assumptions as in the unconstrained case:
\begin{itemize}
    \item [(B1)] The solution set $S_K^*$ of problem \eqref{vi} is nonempty.
    \item  [(B2)] $F: \HH \to \HH$ is $L$-Lipschitz continuous on $K$ for some $L>0.$
    \item [(B3)] $F$ is quasar-monotone on $K$, i.e.,
    \[
    \ip{F(u)}{u - u^*} \ge 0 \qquad \forall u^* \in S_K^*, \quad \forall u\in K.
    \]
    \item[(B4)] $F $ is sequentially weakly continuous on $K$.
\end{itemize}
 %Only the restriction of $F$ to $K$ is needed when the initial points are feasible. Closedness is essential here: it makes $K$ a Hilbert space and ensures that $P_K$ is defined on all of $\HH$.

\begin{Lemma}
%[Exact reduction to an unconstrained problem]
\label{lem:reduction}
The mapping $F_K:K\to K$ defined by $F_K(x):=P_K F(x)$ satisfies
\begin{equation}
 S_K^*=\{u^*\in K:P_K F(u^*)=0\},
 \label{eq:projected-zeros}
\end{equation}
and the projected updates \eqref{PP} are exactly
\begin{equation}
 \begin{aligned}
  u^{k+1}&=u^k-\lambda P_K F(v^k),\\
  v^{k+1}&=u^{k+1}-\lambda P_K F(v^k).
 \end{aligned}
 \label{eq:reduced-updates}
\end{equation}
Moreover, $F_K$ is $L$-Lipschitz.
%and satisfies the solution-relative sign condition in the Hilbert space $K$. 
If $F$ is quasar-monotone on
$K$, then $F_K$ is quasar-monotone on $K$. If $F|_K$ is sequentially
weakly continuous from $K$ into $\HH$, then $F_K$ is sequentially weakly
continuous from $K$ into $K$.
\end{Lemma}

\begin{proof}
Orthogonal decomposition gives a bounded linear projection with
\begin{equation}
 P_K u=u\ (u\in K),\qquad \norm{P_K w}\le\norm{w},\qquad
 \ip{P_K w}{u}=\ip{w}{u}\ (w\in \HH,\ u\in K).
 \label{eq:projection-properties}
\end{equation}
For $u^* \in S_K^*$ and with any $w\in K$, the choices $x=u^*+w \in K$ and $x=u^*-w \in K$ yield $\ip{F(u^*)}{w}=0$. Hence $F(u^*)\in K^\perp$ and $P_K F(u^*)=0$. Conversely, $P_K F(u^*)=0$ implies
$\ip{F(u^*)}{x-z}=\ip{P_K F(u^*)}{x-z}=0$ for every $x\in K$. This proves \eqref{eq:projected-zeros}.

We note that all iterates lie in $K$. Linearity of $P_K$ and $P_K u^k=u^k$
therefore give \eqref{eq:reduced-updates}. For $u,v\in K$,
\[
 \norm{P_K F(u)-P_K F(v)}\le\norm{F(u)-F(v)}\le L\norm{u-v}.
\]
The quasar-monotonicity (A3) is preserved because
\[
 \ip{P_K F(u)}{v-u}=\ip{F(u)}{v-u}\qquad(u,v\in K).
\]
Finally, if $u_n\weakto u$ in $K$, then $u_n\weakto u$ in $\HH$:
for every $w\in \HH$, one has
$\ip{u_n-u}{w}=\ip{u_n-u}{P_K w}\to0$.
Sequential weak continuity of $F|_K$ implies $F(u_n)\weakto F(u)$
in $\HH$. Testing against each $w\in K$ then yields
$\ip{P_K F(u_n)}{w}\to\ip{P_K F(x)}{w}$, as required.\\
\end{proof}
Applying the Theorem \eqref{Mainweakcon} for $F_K :K \to K$ and $0<\lambda < \frac{1}{\sqrt{3}L}$, we obtain the weak convergence for sequences generated by \eqref{PP}.
%\begin{Remark}%[The correct residual is projected]
%\label{rem:projected-residual}
%A constrained solution need not satisfy %$F(u^*)=0$. What vanishes is
%$P_K F(u^*)$. In particular, the estimate needed below is
%\begin{equation}
% \norm{P_K F(y)}
% =\norm{P_K(F(y)-F(u^*))}\le L\norm{y-z}
% \qquad(y\in K,\ u^* \in S^*),
% \label{eq:root-Lipschitz}
%\end{equation}
%not the generally false inequality %$\norm{F(y)}\le L\norm{y-z}$.
%No invariance assumption $F(C)\subseteq C$ is made.
%\end{Remark}

%The reduced updates give, for every %$k\ge1$,
%\begin{align}
% v^k&=u^k-\lambda P_K F(v^{k-1})=2u^k-u^{k-1},
% \label{eq:reflection}\\
% v^k-u^k&=-\lambda P_K F(v^{k-1})=u^k-u^{k-1},
% \label{eq:e-identity}\\
% v^k-v^{k-1}&=(u^k-v^{k-1})+(v^k-u^k),
% \label{eq:delta-identity}\\
% u^{k+1}-v^k&=-\lambda P_K\bigl(F(v^k)-F(v^{k-1})\bigr).
% \label{eq:r-identity}
%\end{align}
%Consequently,
%\begin{equation}
% \norm{u^{k+1}-v^k}^2
% \le\lambda^2L^2\norm{v^k-v^{k-1}}^2\qquad(k\ge1).
% \label{eq:r-bound}
%\end{equation}
%The reflection identity \eqref{eq:reflection} is the decisive point:
%although $K$ may be a proper subset of $\HH$, it survives exactly for linear-subspace projections. 
%Thus the identity singled out in \cite[Remark~2.1]{NTV2026} remains available here. All quantities $A_k$ below start at $k=2$, so no negative-index iterate is required.
}

\section{Conclusion}
We establish a new and tight upper bound on stepsize for the Popov's algorithm and its variants for solving (un)constrained variational inequalities. The results obtained provide  answers to the open questions posed in \cite{Vuong,Xu}.

\end{document}